\documentclass[12pt]{article}

\usepackage[colorinlistoftodos]{todonotes}

\usepackage{amssymb}
\usepackage{amstext}
\usepackage{amsgen}
\usepackage{amsxtra}
\usepackage{amsgen}
\usepackage{amsmath}
\usepackage{amsfonts,amsthm, esint}
\usepackage[english]{babel}

\newcommand{\REV}[1]{{#1}} 
\newcommand{\REVV}[1]{{#1}} 
\newcommand{\REVTWO}[1]{{\color{blue}#1}} 

\newtheorem{theorem}{Theorem}[section]
\newtheorem{lemma}[theorem]{Lemma}
\newtheorem{e-proposition}[theorem]{Proposition}
\newtheorem{corollary}[theorem]{Corollary}
\newtheorem{assumption}{Assumption}
\newtheorem{proposition}[theorem]{Proposition}
\newtheorem{e-definition}[theorem]{Definition\rm}
\newtheorem{remark}{\it Remark\/}

\def\eps{\varepsilon}

\def\RR{\mathbb{R}}

\def\cC{\mathcal{C}}
\def\cP{\mathcal{P}}

\newcommand{\N}{\mathbb{N}}

\newcommand{\R}{\mathbb{R}}

\def\og{\leavevmode\raise.3ex\hbox{$\scriptscriptstyle\langle\!\langle$~}}
\def\fg{\leavevmode\raise.3ex\hbox{~$\!\scriptscriptstyle\,\rangle\!\rangle$}}

\title{Uniqueness of strong solutions and weak-strong stability in a system of cross-diffusion equations}
\author{Judith Berendsen$^{1}$, Martin Burger$^{2}$, Virginie Ehrlacher$^{3,4}$, Jan-Frederik Pietschmann$^{1,*}$
\\
{\footnotesize $^1$ Fakult\"at f\"ur Mathematik, Technische Universit\"at Chemnitz, Reichenhainer Stra\ss{}e 41, Germany}
\\
{\footnotesize $^2$ Department Mathematik, Friedrich-Alexander Universit\"at Erlangen-N\"urnberg, Cauerstrasse 9, 91058 Erlangen, Germany}
\\
{\footnotesize $^3$ CERMICS, \'Ecole des Ponts ParisTech, 77455 Marne-La-Vall\'ee Cedex 2, France}
\\
{\footnotesize $^4$ Inria Paris, MATHERIALS project-team, 2 rue Simone Iff, CS 42112, 75589 Paris Cedex 12, France}
\\
{\footnotesize $^*$ \textit{Corresponding author}. Email: jfpietschmann@math.tu-chemnitz.de}
}
\date{\today}

\begin{document}

\maketitle
\selectlanguage{english}

\begin{abstract}
Proving the uniqueness of solutions to multi-species cross-diffusion systems is a difficult task in the general case, and there exist very few results in this direction. 
In this work, we study a particular system with zero-flux boundary conditions for which the existence of a weak solution has been proven in~\cite{Ehrlacher2017}. 
Under additional assumptions on the value of the cross-diffusion coefficients, we are able to show the existence \REVV{and uniqueness of non-negative} strong solutions. 
The proof \REVV{of the existence }relies on the use of an appropriate \REVV{linearized problem} and a fixed-point argument. In addition, a weak-strong stability result is obtained for this system in dimension one which also implies uniqueness 
\REVV{of weak solutions}.
\end{abstract}

\section{Introduction}
Systems of partial differential equations with cross-diffusion have gained a lot of interest in recent years \cite{Kuefner1996,Chen2004,Chen2006,Lepoutre2012,Juengel2015bddness}. Such systems
appear in many applications, for instance the modelling of population dynamics of multiple species~\cite{Burger2015}, 
cell sorting or chemotaxis-like applications~\cite{Painter2002,Painter2009} or predator-swarm model~\cite{Chen2014}, and have been studied in different contexts~\cite{Aman1990,Le2006,Griepentrog2010,Griepentrog2004,chapman1,chapman2}. 

\medskip

In this work we focus our attention to a particular multi-species cross-diffusion system which reads as follows. Let $T>0$, $n,d\in \N^*$ and $\Omega \subset \R^d$ be a bounded regular domain. For $t\in (0,T)$ and $x\in \Omega$, 
we consider $(u_0(t,x),\ldots,u_n(t,x))$ to be a solution to the system of $n+1$ equations
\begin{equation}\label{eq:sys}
    \partial_t u_i - \nabla \cdot \left[\sum_{j=0, j\neq i}^n K_{ij} (u_j\nabla u_i-u_i\nabla u_j)\right] = 0  \text{ in }\quad (0,T) \times \Omega,\quad i=0,\ldots, n,
\end{equation}
supplemented with no-flux boundary conditions and some initial data, where $K_{ij}\geq 0$ for all $0\leq i \neq j \leq n$.

System~\eqref{eq:sys} can be seen as the (formal) limit of a microscopic stochastic lattice based model (see the Appendix for more details) and models the evolution of the local volumic fractions of a system composed of $n+1$ different species. The function $u_i(t,x)$ represents the value at some 
time $t\in [0,T]$ and point $x\in \Omega$ of the density or volumic fraction of the $i^{th}$ entity. 
In terms of modelling this means that the particles whose densities are given by the functions $u_i$ have a finite, positive size so that there is a maximal number of particles per given volume. 
This is often referred to as size exclusion (or exclusion process). From a modelling point of view, one is therefore interested in considering solutions to (\ref{eq:sys}) which satisfy
\begin{equation}\label{eq:volume}
\forall 0\leq i \leq n, \quad u_i(t,x)\geq 0 \quad \mbox{ and } \quad \sum_{i=0}^n u_i(t,x) =1, \quad \mbox{ a.e. in }(0,T)\times \Omega.  
\end{equation}

\medskip

From an analysis point of view, it is not easy to prove the existence of solutions to cross-diffusion systems satisfying (\ref{eq:volume}), and uniqueness results are even harder to obtain. 
Recently, cross-diffusion systems which exhibit a (formal) gradient flow structure (see~\cite{Juengel2015},~\cite{Burger2010}, or~\cite{Schlake2011}) have drawn particular interest from mathematicians. 
Indeed, such a structure allows to show the existence of weak solutions in many situations, using the dissipation of the corresponding entropy to get a priori bounds which are enough to pass to the limit in a suitable approximation. 
This often also relies on the introduction of so-called entropy variables which can be used as a substitute for maximum principles which are not available for such systems. 
See e.g. \cite{Juengel2012}, \cite{Juengel2015} and \cite{Ehrlacher2017} for examples of this strategy. Also note that due to the degenerate structure of \eqref{eq:sys}, solutions sometimes have less regularity than in the usual parabolic case (e.g. for $n=2,K_{10}=K_{20}=1$ and $K_{12}=K_{21}=0,$ the solutions $u_i$ are only $L^2$ in space, not $H^1$, see \cite{Burger2010} for details). 

\medskip

The existence of weak solutions to (\ref{eq:sys}) is proved in~\cite{Ehrlacher2017}, using a general result of~\cite{Juengel2015bddness}, under the assumption that the 
cross-diffusion coefficients $(K_{ij})_{0\leq i \neq j \leq n}$ are assumed to be positive and to satisfy $K_{ij} = K_{ji}$ for all $0\leq i \neq j \leq n$. 
Most importantly in the context of our work, the existence of strong solutions and uniqueness of weak solutions was
so far only available in a very special situation, i.e. when all the self-diffusion coefficients $K_{ij}$ are equal to a constant $K$. In this case, system~\eqref{eq:sys} boils down to a system of $n+1$ independent heat equations, 
the analysis of which is easy. In the general case, to the best of our knowledge, no such results are available so far.

\medskip

The object of the present article is based on the following observation: The system \eqref{eq:sys} can be considered as a perturbation of a system of heat equations if the coefficients $K_{ij}$ are not too different 
from a fixed constant $K$. Indeed, we have 
\begin{align}\label{eq:heatperturbation}
 \partial_t u_i - K \Delta u_i = \mbox{\rm div} \left[ \sum_{j=0 }^n (K_{ij}-K) (u_j\nabla u_i-u_i\nabla u_j)\right]. 
\end{align}
Under the assumption that the quantities $\vert K_{ij} - K \vert$ are sufficiently small, we prove the existence \REVV{and uniqueness} of strong solutions to \eqref{eq:sys} \REVV{satisfying (\ref{eq:volume}) in dimension $d\leq 3$}. 
\REVV{In addition, we prove} a weak-strong stability estimate 
in dimension 1 which implies the uniqueness 
of weak solutions. A key issue in the proof is to construct approximations that preserve nonnegativity and the volume constraint.

This paper is organized as follows: in Section~\ref{sec:main} we state our main results. 
Section~\ref{sec:proofex} is devoted to the proof of the existence \REVV{ and uniqueness} of strong solutions to the cross-diffusion system we consider \REVV{in dimension $d\leq 3$}. 
Lastly, Section~\ref{sec:stability} details the proof of the weak-strong stability result we obtain in dimension 1. Let us mention that a weak-strong stability result is proved in~\cite{Chen2018} 
for a system similar to (\ref{eq:sys}), but with different assumptions on the coefficients $(K_{ij})_{0\leq i\neq j \leq n}$.

\section{Main results}\label{sec:main}

\subsection{Notation and preliminaries}

 Let $T>0$, $n,d\in \N^*$ and $\Omega \subset \R^d$ be a bounded regular domain. For $t\in (0,T)$ and $x\in \Omega$, 
we consider
$u(t,x):=(u_0(t,x),\ldots,u_n(t,x))$ solution to the system of $n+1$ equations:
\begin{equation}\label{eq:crossnplusone}
  \left\{
    \begin{array}{rll}
    \partial_t u_i - \nabla \cdot \left[\sum_{j=0, j\neq i}^n K_{ij} (u_j\nabla u_i-u_i\nabla u_j)\right]& = 0  & \text{ in }\quad (0,T) \times \Omega,\\
    \left[\sum_{j=0, j\neq i}^n K_{ij}(u_j\nabla u_i-u_i\nabla u_j)\right] \cdot \textbf{n} & = 0 & \text{ on }\quad (0,T) \times \partial\Omega, \\
  \end{array} \right.
  \quad i=0,\ldots, n,
\end{equation}
where $\textbf{n}$ denotes the unit outward pointing normal to the domain $\Omega$, and $(K_{ij})_{0\leq i \neq j \leq n}$ are non-negative coefficients. System~\eqref{eq:crossnplusone} is supplemented with the initial condition $u^0:=(u_0^0,\ldots,u_n^0)\in \left(L^1(\Omega)\right)^{n+1}$ 

\medskip

We make the following assumption on the values of the cross-diffusion coefficients $(K_{ij})_{0\leq i \neq j \leq n}$: 
\begin{assumption}\label{ass:A1}
For all $ 0 \leq i \neq j \leq n$, $K_{ij}>0$ and $K_{ij} = K_{ji}$.
\end{assumption}

\medskip

As mentioned in the introduction, such a system models the evolution of the local volumic fractions of a system composed of $n+1$ different species and we expect the nonnegativity and volume constraint \eqref{eq:volume} to hold.

\medskip

Let us denote by
$$
\mathcal P:= \left\{ u:=(u_0,\ldots,u_n) \in (0,+\infty)^{n+1}, \sum_{i=0}^n u_i = 1\right\} \; \mbox{ and } \; \mathcal D:= \left\{ U:=(u_1,\ldots,u_n) \in (0,+\infty)^{n}, \sum_{i=1}^n u_i < 1\right\}.
$$
We point out that for all $u:=(u_0,\ldots, u_n)\in \mathbb{R}^{n+1}$, $u$ belongs to $\mathcal P$ if and only if $U:=(u_1,\ldots,u_n)$ belongs to $\mathcal D$. Similarly, $u$ belongs to $\overline{\mathcal P}$ if and only if 
$U$ belongs to $\overline{\mathcal{D}}$. Let us mention that condition (\ref{eq:volume}) can be equivalently rewritten as $u(t,x)\in \overline{\mathcal P}$ for almost all $(t,x)\in (0,T)\times \Omega$.  
In what follows, we assume that the initial condition $u^0$ satisfies the following constraint:
\begin{equation}\label{ass:init}
 u^0(x) \in \overline{\mathcal P} \mbox{ for almost all }x\in\Omega.  
\end{equation}
Under Assumptions~\ref{ass:A1} and~\ref{ass:init}, it is easy to see (at least formally) that the dynamics of the system preserves the volume constraint, i.e. 
\begin{equation}\label{eq:volume2}
 \sum_{i=0}^n u_i(t,x)=1  \quad  \text{ a.e. in }(0,T) \times \Omega.
\end{equation}
However, proving the existence of (weak or strong) solutions to system~(\ref{eq:crossnplusone}) so that $u_i(t,x)\geq 0$ for all $0\leq i \leq n$ and almost all $(t,x)\in (0,T)\times \Omega$ is an intricate task 
from an analysis point of view. \\
The existence of weak solutions to system~(\ref{eq:crossnplusone}) satisfying (\ref{eq:volume}) is proved in~\cite{Ehrlacher2017} under Assumptions~\ref{ass:A1} and~\ref{ass:init} and is actually 
a consequence of Theorem~2 of~\cite{Juengel2015bddness}. Let us recall this result and the main arguments of its proof below.
Using (\ref{eq:volume2}) to express $u_0$ as $1 - \sum_{i=1}^n u_i$, system (\ref{eq:crossnplusone}) can be equivalently rewritten as a system of $n$ equations of the form
\begin{equation}\label{eq:cross}
\left\{
  \begin{array}{rll}
    \partial_t U - \nabla \cdot (A(U) \nabla U) &= 0 & \text{ on }\quad (0,T) \times \Omega,\\
    (A(U)\nabla U) \cdot \textbf{n}&= 0 & \text{ on }\quad (0,T) \times \partial\Omega,\\
  \end{array}
  \right.  
\end{equation}
where $U := (u_1, \ldots, u_n)$. The diffusion matrix $A$ is defined by
\begin{equation}\label{eq:defA1}
A: \left\{
\begin{array}{ccc}
\R^n & \to & \R^{n\times n}\\
U & \mapsto & A(U) := \left( A_{ij}(U)\right)_{1\leq i,j \leq n},\\
\end{array}
\right.
\end{equation}
with, for all $U:=(u_1,\ldots,u_n)\in \R^n$, 
\begin{equation}\label{eq:defA}
\begin{array}{lll}
A_{ii}(U)&= \sum_{j=1, j\neq i}^n (K_{ij}-K_{i0})u_j+K_{i0,} & \quad i=1,\ldots n,\\
A_{ij}(U)&= -(K_{ij}-K_{i0})u_i, & \quad i,j=1,\ldots, n,\; i \neq j.\\
\end{array}
\end{equation}

%

Theorem~2 of~\cite{Juengel2015bddness} gives sufficient conditions on the diffusion matrix $A$ for a general cross-diffusion system to have a weak solution so that $U(t,x)\in \overline{\mathcal{D}}$ for almost all $(t,x)\in (0,T)\times \Omega$. 
More precisely, Theorem~2 of~\cite{Juengel2015bddness} can be stated as follows.

\begin{theorem}[Theorem~2 of~\cite{Juengel2015bddness}]\label{th:Jungel}
Let $A \in \mathcal{C}^0(\overline{D}; \R^{n\times n})$ be a continuous matrix-valued field defined on $\overline{\mathcal{D}}$ satisfying the following assumptions:
\begin{itemize}
 \item[(H1)] there exists a bounded from below convex function $h\in \mathcal{C}^2(\mathcal{D}, \R)$ such that its derivative $Dh: \mathcal{D} \to \R^n$ is invertible on $\R^n$;
 \item[(H2)] there exists $\alpha >0$, and for all $1\leq i \leq n$ there exists $1\geq m_i >0$ such that, for all $z:=(z_1,\ldots,z_n)\in \R^n$ and all $U:=(u_1,\ldots,u_n)\in\mathcal{D}$, 
 $$
 z^T D^2h(U)A(U) z \geq \alpha \sum_{i=1}^n u_i^{2m_i-2}z_i^2.
 $$
\end{itemize}
Let $U^0\in L^1(\Omega; \mathcal{D})$ such that $w^0:= Dh(U^0) \in L^\infty(\Omega; \R^n)$. Then, there exists a weak solution $U$ with initial condition $U^0$ to
\begin{equation}\label{eq:crossdiffgen}
\left\{
  \begin{array}{rll}
    \partial_t U - \nabla \cdot (A(U) \nabla U) &= 0 & \text{ on }\quad (0,T) \times \Omega,\\
    (A(U)\nabla U) \cdot \textbf{n} &= 0 & \text{ on }\quad (0,T) \times \partial\Omega,\\
  \end{array}
  \right.
\end{equation}
such that for almost all $(t,x)\in (0,T)\times \Omega$, $U(t,x)\in \overline{D}$ with 
$$
U\in L^2_{\rm loc}((0,T); H^1(\Omega; \R^n)) \quad \mbox{ and } \partial_t U \in L^2_{\rm loc}((0,T); (H^1(\Omega;\R^n))').
$$
\end{theorem}

The proof of Theorem~\ref{th:Jungel} relies on the fact that a system of the form~\eqref{eq:crossdiffgen} with a matrix-valued field $A$ satisfying conditions (H1)-(H2)
exhibits a (formal) gradient flow structure, which we detail below for the particular case of~\eqref{eq:cross} with $A$ defined by~\eqref{eq:defA}. 

To this end, let us introduce the entropy density $h$ given by
\begin{equation}\label{eq:defh}
   h: \left\{
   \begin{array}{ccc}
    \bar{\mathcal{D}} & \to & \mathbb{R} \\
    U:= (u_i)_{1\leq i \leq n} & \mapsto & \sum_{i=1}^n u_i \log u_i + (1-\sum_{i=1}^n u_i)\log (1-\sum_{i=1}^n u_i).\\
   \end{array} \right.
\end{equation}
and the corresponding entropy functional
\begin{equation}\label{eq:defE}
   \mathcal{E}:
   \begin{array}{ccc}
 L^{\infty}(\Omega,\bar{\mathcal{D}}) & \to & \mathbb{R} \\   
     U & \mapsto&  \int_{\Omega} h(U(x))\, dx .\\
   \end{array}
\end{equation}
It is proved in Lemma~2.3 of~\cite{Ehrlacher2017} that the function $h$ defined by (\ref{eq:defh}) satisfies conditions (H1) and (H2) of Theorem~\ref{th:Jungel} for the matrix-valued function $A$ defined by (\ref{eq:defA}) 
with $m_i = \frac{1}{2}$ for every $1\leq i  \leq n$ and $\alpha=\min_{1\leq i\neq j\leq n}K_{ij}$. Furthermore, we can rewrite the system \eqref{eq:cross} as
\begin{align*}
\partial_t U - \nabla \cdot (A(U)(D^2h(U))^{-1} \nabla D\mathcal{E}(U)) &= 0 \text{ on }\quad (0,T) \times \Omega, \\
(A(U)(D^2h(U))^{-1}\nabla D\mathcal{E}(U)) \cdot \normalfont\textbf{n} &= 0 \text{ on }\quad (0,T) \times \partial\Omega, \\
U(0,x)&=U^0(x) \quad \text{ a.e. in } \Omega.
\end{align*}
In this formulation, it becomes clear that the entropy functional $\mathcal{E}$ is a Lyapunov function for system (\ref{eq:cross}).

\medskip

The existence of weak solutions to~\eqref{eq:cross} satisfying $U(t,x)\in \overline{\mathcal D}$ almost everywhere is then a direct consequence of Theorem~2 of \cite{Juengel2015bddness} and Lemma~2.3 of~\cite{Ehrlacher2017}. More precisely, we have the following proposition
\begin{proposition}[Existence of weak solutions]\label{prop:weak}
 Let $u^0\in L^1(\Omega; \mathcal{P})$ and $U^0:=(u_1^0, \ldots, u_n^0)$. Let us assume in addition that $w^0:= Dh(U^0) \in L^\infty(\Omega; \R^n)$ with $h$ defined by (\ref{eq:defh}). 
 Then, there exists a weak solution $u$ with initial condition $u^0$ to \eqref{eq:crossnplusone} such that
 \begin{itemize}
  \item[(i)] $\displaystyle u\in L^2_{\rm loc}((0,T); H^1(\Omega; \R^{n+1}))$ and $\displaystyle \partial_t u \in L^2_{\rm loc}((0,T); (H^1(\Omega;\R^{n+1}))')$;
  \item[(ii)] for almost all $(t,x)\in (0,T)\times \Omega$, $u(t,x)\in \overline{\cP}$. 
 \end{itemize}
\end{proposition}

\subsection{Main results}

The aim of this work is to prove the existence \REVV{ and uniqueness} of strong solutions to system (\ref{eq:cross}) satisfying (\ref{eq:volume}) under additional assumptions on the cross-diffusion coefficients $(K_{ij})_{0\leq i \neq j \leq n}$. 
Such a result holds \REVV{ in dimension $d\leq 3$}. For the particular case when $d=1$, we can also prove a weak-strong stability result which 
implies that there exists a unique weak solution to system~(\ref{eq:crossnplusone}) satisfying (\ref{eq:volume}) and that this solution is strong.

Before stating our main results, let us make a preliminary remark on the no-flux boundary conditions imposed on $U$ in (\ref{eq:cross}) which will be useful in the sequel. It is shown in \cite[Lemma~5]{Juengel2012} that the matrix 
$A(U)$ is invertible for all $U\in \overline{\mathcal{D}}$. Besides, for all $U\in L^1(\Omega; \overline{\mathcal{D}})$, $(A(U) \nabla U)\cdot \textbf{n} = A(U) \left( \nabla U \cdot \textbf{n}\right)$ on $\partial \Omega$. 
This implies that a solution 
$U$ to~\eqref{eq:cross} is equivalently a solution to the system
\begin{equation}\label{eq:cross2}
\left\{
  \begin{array}{rll}
    \partial_t U - \nabla \cdot (A(U) \nabla U) &= 0 & \text{ on }\quad (0,T) \times \Omega,\\
    \nabla U \cdot \textbf{n}&= 0 & \text{ on }\quad (0,T) \times \partial\Omega,\\
  \end{array}
  \right.
\end{equation}
and, denoting by $u_0: = 1 - \sum_{i=1}^n u_i$, $u:=(u_0, \ldots, u_n)$ is then equivalently a solution to
\begin{equation}\label{eq:crossnplusone2}
  \left\{
    \begin{array}{rll}
    \partial_t u_i - \nabla \cdot \left[\sum_{j=0, j\neq i}^n K_{ij} (u_j\nabla u_i-u_i\nabla u_j)\right]& = 0  & \text{ in }\quad (0,T) \times \Omega,\\
    \nabla u_i \cdot \textbf{n} & = 0 & \text{ on }\quad (0,T) \times \partial\Omega, \\
  \end{array} \right.
  \quad i=0,\ldots, n.
\end{equation}
Proving the existence of strong solutions to system~(\ref{eq:cross}) is then equivalent to proving the existence of strong solutions to system~\eqref{eq:crossnplusone2} and it will be more convenient for our analysis to consider the latter formulation in the sequel. 

To obtain this strong existence result, we make an additional assumption on the cross-diffusion coefficients $(K_{ij})_{0\leq i \neq j \leq n}$ which we detail hereafter.  For all 
$0\leq i \leq n$, let 
\REVTWO{
\begin{equation}\label{eq:defKi}
K^+:= \max_{0\leq j \neq i \leq n} K_{ij}, \quad K^-:= \min_{0\leq j \neq i \leq n} K_{ij}, \quad K:= \frac{K^+ + K^-}{2} \quad \mbox{ and } \kappa:= \frac{K^+ - K^-}{2}. 
\end{equation}
}

The additional assumption which we make from now on and in all the sequel reads as follows: 

\begin{assumption}\label{ass:A2}
$\displaystyle K \; > \; 2n\kappa$. 
\end{assumption}

\medskip

In other words, Assumption~\ref{ass:A2} means that all the coefficients $K_{ij}$ should be sufficiently close to one another. The motivation for considering such a situation stems from the following observation: if there exists a constant $K>0$ such that 
for all $0\leq i \neq j \leq n$, $K_{ij} = K$, then $\kappa = 0$ and system~(\ref{eq:crossnplusone2}) boils down to a system of $n+1$ independent heat equations for which the existence and uniqueness of strong solutions satisfying (\ref{eq:volume}) is obvious. \REVTWO{Let us point out direct consequences of Assumption \ref{ass:A2} that will be used frequently. We have 
	\begin{align}
	 |K_{ij} -K| \le \kappa\quad \text{ as well as }\quad K- 2n\kappa - \eps > 0
	\end{align}
	for $\eps > 0$ sufficiently small.
}
%
%

\medskip

We are now in position to state our two main results. 


\REVTWO{This lemma ensures that provided a strong solution exists, it is unique. Existence is dealt with in the following result.}
\begin{theorem}(Existence and uniqueness of strong solutions)\label{thm:existencestrong}
Let us assume \REVV{that $d\leq 3$ and} that Assumptions~\ref{ass:A1} and~\ref{ass:A2} hold. Then, for every initial datum $u^0\in [H^1(\Omega)]^{n+1}$, with $u^0(x) \in\overline{\mathcal{P}}$ for almost all $x\in\Omega$, 
there exists a \REV{unique} strong solution $u$ to \eqref{eq:crossnplusone2} (or equivalently to \eqref{eq:crossnplusone}) such that
\begin{itemize}
 \item [(i)] $\displaystyle u \in [L^2((0,T), H^2(\Omega)) \cap H^1((0,T), L^2(\Omega))]^{n+1}$;
 \item[(ii)] $u(t,x)\in \overline{\cP}$ for almost all $(t,x)\in (0,T)\times \Omega$.
\end{itemize}
\end{theorem}

\begin{theorem}(Weak-strong stability estimate in $d=1$)\label{thm:stability}
Let us assume that Assumptions~\ref{ass:A1} and~\ref{ass:A2} hold. Let $\tilde u$ be a weak solution to~\eqref{eq:crossnplusone2} (or equivalently to \eqref{eq:crossnplusone}) in the sense of 
Proposition~\ref{prop:weak}, and let $u$ be a strong solution in the sense of 
Theorem~\ref{thm:existencestrong}.
Then, there exists a constant $C>0$ such that the following stability estimate holds for all $0< t\leq T$:
\begin{equation}\label{eq:estimate}
\|u(t,\cdot)-\tilde{u}(t,\cdot)\|_{L^2(\Omega)}^2 \leq e^{C\| \nabla u\|_{L^2(0,t; L^\infty(\Omega))}^2}\|u(0,\cdot)-\tilde{u}(0,\cdot)\|_{L^2(\Omega)}^2.
\end{equation}
\end{theorem}

A direct corollary of Theorem~\ref{thm:stability} is the weak-strong uniqueness of solutions to \eqref{eq:crossnplusone2} in dimension~1 for regular initial data, which can be stated as follows.
\begin{corollary}\label{cor:wsuniq}
Let us assume that $d=1$ and let $u^0 \in H^1(\Omega)^{n+1}$ such that $u^0(x)\in\overline{\mathcal P}$ for almost all $x\in \Omega$.
Let $\tilde u$ be a weak solution to \eqref{eq:crossnplusone2} (or equivalently to \eqref{eq:crossnplusone}) in the sense of Proposition \ref{prop:weak} and $u\in L^2((0,T), H^2(\Omega)) \cap H^1((0,T), L^2(\Omega))$ 
be a strong solution in the sense of Theorem \ref{thm:existencestrong}.\\ 
Then, if the corresponding initial data $u(0, \cdot)$ and $\tilde{u}(0,\cdot)$ agree a.e. on $\Omega$, we also have
\begin{align*}
u = \tilde u \text{ a.e. in } \Omega \times (0,T).
\end{align*}
\end{corollary}

\REVV{For the proof of Theorem~\ref{thm:stability} and Corollary~\ref{cor:wsuniq}}, we are restricted to one spatial dimension due to the fact that we need the embedding $H^2(\Omega) \hookrightarrow W^{1,\infty}(\Omega)$, 
so that a strong solution $u \in \left[L^2((0,T), H^2(\Omega)) \cap H^1((0,T), L^2(\Omega))\right]^{n+1}$ in the sense of Theorem~\ref{thm:existencestrong}, satisfies $\nabla u \in \left[L^2(0,T; L^\infty(\Omega))^d\right]^{n+1}$.

\section{Proof of Theorem~\ref{thm:existencestrong}}\label{sec:proofex}
\selectlanguage{english}
The aim of this section is to prove Theorem~\ref{thm:existencestrong} which states the existence \REVV{and uniqueness} of strong solutions to system~(\ref{eq:crossnplusone2}).

The proof is based on a fixed-point argument: we will show existence and uniqueness of strong solutions to a linearized system and subsequently apply Brouwer's fixed point theorem. 
\REV{Uniqueness will be shown seperately}. For convenience let us define
\begin{align}\label{eq:def_W}
W:= L^2((0,T);H^2(\Omega)) \cap H^1((0,T);L^2(\Omega)),
\end{align}
\REVV{
and 
$$
W_n:= \left\{ u \in W, \; \nabla u \cdot \textbf{n} = 0 \mbox{ a.e. on } (0,T)\times \partial \Omega\right\}. 
$$
It can be easily checked that $W_n$ is a Banach space when endowed with the norm defined by
$$
\forall u\in W_n, \quad \|u\|^2_{W_n}:= \|\partial_t u \|^2_{L^2(0,T;L^2(\Omega))} + \|\Delta u \|^2_{L^2(0,T; L^2(\Omega))} + \|u(t=0,\cdot)\|^2_{H^1(\Omega)}.
$$
}

\normalfont

\subsection{Auxiliary \REV{lemmata}}

In this section, we start by stating some auxiliary results which are needed in the sequel.

\begin{lemma}\label{lem:timederiv} 
Let $u \in W$. Then, it holds that 
\begin{itemize}
 \item[(i)] $u \in \cC([0,T];H^1(\Omega))$;
 \item[(ii)] there exists a constant $C>0$ which only depends on $T$ and $\Omega$ such that
\begin{equation}\label{eq:est1}
 \mathop{\max}_{0\leq t \leq T} \|u(t,\cdot)\|_{H^1(\Omega)} \leq C\left( \|u\|_{L^2(0,T; H^2(\Omega))} + \|\partial_t u\|_{L^2(0,T; L^2(\Omega)}\right);
\end{equation}
\item[(iii)] if in addition \REVV{$u\in W_n$} (i.e. if $\nabla u \cdot \textbf{n} = 0$ a.e. in $\partial \Omega \times (0,T)$), the mapping $(0,T)\ni t \mapsto \|\nabla u(t, \cdot)\|_{L^2(\Omega)}^2$ is absolutely continuous, with 
$$
\frac{d}{dt}\|\nabla u(t, \cdot)\|_{L^2(\Omega)}^2 = -2 \langle \partial_t u(t,\cdot), \Delta u(t,\cdot) \rangle_{L^2(\Omega)}, \quad \mbox{ for a.e. }t \in (0,T).
$$

\normalfont
\end{itemize}
\end{lemma}

\begin{proof}
Items (i) and (ii) are direct applications of [Theorem~4, Section~5.9.2] of~\cite{Evans}. Let us now turn to the proof of (iii). 
We extend $u$ by zero to a function $\overline{u}$ defined for all $t \in \RR$ and define, for all $\delta >0$ and $t\in \R$, $u^\delta(t,\cdot) := \int_\R \eta_\delta(t-s)\overline{u}(s,\cdot)\,ds$, 
where $\eta_\delta$ is a standard mollifier. It then holds that $u^\delta \in \cC^\infty(0,T; H^2(\Omega))$ and that for all $t\in (0,T)$, $\nabla u^\delta(t, \cdot) \cdot \textbf{n} = 0$ a.e. in $\partial \Omega$.

Let $0< t <T$. Then, we have
\begin{align*}
\frac{d}{dt} \left( \|\nabla u^\delta(t,\cdot)\|_{L^2(\Omega)}^2\right) = 2\langle\partial_t \nabla u^\delta(t,\cdot), \nabla u^\delta(t, \cdot)\rangle_{L^2(\Omega)} = 2\int_\Omega \partial_t\nabla u^\delta(t,\cdot) \cdot \nabla u^\delta(t,\cdot),
\end{align*}
where $\partial_t \nabla u^\delta$ is the weak time derivative of $\nabla u^\delta$. For all $0\leq t \leq T$, the following convergence holds strongly in $L^2(\Omega)$
$$
\mathop{\lim}_{h\to 0 }\frac{\nabla u^\delta (t+h,\cdot) - \nabla u^\delta(t,\cdot)}{h} = \partial_t \nabla u^\delta(t,\cdot), 
$$
and
\begin{equation}\label{eq:boundunif}
\left\|\frac{\nabla u^\delta (t+h, \cdot) - \nabla u^\delta(t, \cdot)}{h}\right\|_{L^2(\Omega)} = \left\|\frac{1}{h}\int_{t}^{t+h} \partial_s \nabla u^\delta(s, \cdot)\,ds\right\|_{L^2(\Omega)} \leq \mathop{\sup}_{0 \leq s \leq T} \|\partial_s \nabla u^{\delta}(s, \cdot)\|_{L^2(\Omega)}.
\end{equation}
Inequality (\ref{eq:boundunif}) implies that the difference quotient $\displaystyle \left( \frac{\nabla u^\delta (\cdot,t+h) - \nabla u^\delta(\cdot,t)}{h}\right)_{h>0}$ is uniformly bounded in $L^2(\Omega)$ as $h$ goes to $0$, so that 
$\displaystyle \left( \frac{\nabla u^\delta (\cdot,t+h) - \nabla u^\delta(\cdot,t)}{h}\cdot \nabla u^\delta \right)_{h>0}$ is uniformly bounded in $L^1(\Omega)$. As a consequence, applying Lebesgue's dominated convergence theorem, we obtain
\begin{align*}
 2\int_\Omega \partial_t\nabla u^\delta(t,\cdot) \cdot \nabla u^\delta(t,\cdot) & = \mathop{\lim}_{h\to 0} \int_\Omega  \frac{\nabla u^\delta (\cdot,t+h) - \nabla u^\delta(\cdot,t)}{h}\cdot \nabla u^\delta(\cdot, t).
\end{align*}
Besides, since $\nabla u^\delta(s, \cdot) \cdot \textbf{n} = 0$ on $\partial \Omega$ for all $s \in (0,T)$, it holds that for all $h>0$, 
$$
\int_\Omega  \frac{\nabla u^\delta (\cdot,t+h) - \nabla u^\delta(\cdot,t)}{h}\cdot \nabla u^\delta = - \int_\Omega \frac{u^\delta (\cdot,t+h) - u^\delta(\cdot,t)}{h} \Delta u^\delta(\cdot, t).
$$
Applying again Lebesgue's convergence theorem, we obtain
$$
\mathop{\lim}_{h\to 0} \int_\Omega \frac{u^\delta (\cdot,t+h) - u^\delta(\cdot,t)}{h} \Delta u^\delta(\cdot, t) = \int_\Omega \partial_t u^\delta (\cdot,t) \Delta u^\delta(\cdot, t), 
$$
in $L^1(\Omega)$. Thus, for all $\delta >0$, we have
\begin{equation}\label{eq:delta}
\frac{d}{dt} \left( \|\nabla u^\delta(t,\cdot)\|_{L^2(\Omega)}^2\right) = - 2\langle \partial_t u^\delta(t,\cdot), \Delta u^\delta(t,\cdot)\rangle_{L^2(\Omega)}.  
\end{equation}
As $\delta$ goes to $0$, the convergences $u^\delta \to u,\, \nabla u^\delta \to \nabla u$ and $\Delta u^\delta \to \Delta u$ hold strongly in $L^2(0,T; L^2(\Omega))$ (since $u \in L^2((0,T);H^2(\Omega))$.
We finally obtain the result by passing to the limit $\delta \to 0$ in (\ref{eq:delta}).
\end{proof}

\REVV{
Let us now introduce the space
$$
W_0 := \{ u \in W\; | \; \nabla u \cdot \textbf{n} = 0 \text{ a.e. on } \partial\Omega \times (0,T),\, u=0 \text{ a.e. on } \Omega \times \{t=0\} \},
$$
which is a closed subspace of $W_n$. For all $u \in W_0$ and all $0\leq i \leq n$, let us define
$$
\| u \|_{W_0}^2 :=\int_0^T \| \partial_t u(\cdot,t) - K \Delta u(\cdot,t)\|_{L^2(\Omega)}^2\;dt
$$
and 
$$
\| u \|_{\widetilde W_0}^2 :=\int_0^T \|\partial_t u(\cdot,t)^2\|_{L^2(\Omega)}^2 + K^2 \|\Delta u(\cdot,t)\|_{L^2(\Omega)}^2\;dt.
$$
}

We then have the following result (see also \cite{Breden2017} for a similar argument).
\begin{lemma}[Equivalence of norms]\label{lem:equi}
The two applications $\|\cdot\|_{W_0}$ and $\|\cdot\|_{\widetilde W_0}$ define norms on $W_0$ which are equivalent, and equivalent to the norm $\|\cdot\|_{W_n}$.
\end{lemma}
\begin{proof}
Let $u\in W_0$. On the one hand, 
\begin{align} \nonumber
\|u \|_{W_0}^2 & = \int_0^T\int_\Omega (\partial_t u - K \Delta u)^2 = \int_0^T\int_\Omega (\partial_t u)^2 -2K\partial_t u \Delta u + K^2( \Delta u)^2\\ \label{eq:normestimate}
&\ge \int_0^T\int_\Omega (\partial_t u)^2 + K^2( \Delta u)^2 = \|u\|_{\widetilde W_0}^2\\\nonumber
\end{align}
since
\begin{align*}
\int_0^T\int_\Omega 2K\partial_t u \Delta u = - \int_0^T\int_\Omega 2K\partial_t \nabla u \cdot \nabla u = - K \int_\Omega |\nabla u|^2(T) \le 0.
\end{align*}
On the other hand, we always have 
\begin{equation}
\|u \|_{W_0}^2 = \int_0^T\int_\Omega (\partial_t u - K \Delta u)^2 \le 2\int_0^T\int_\Omega (\partial_t u)^2 + K^2( \Delta u)^2= 2 \|u\|_{\widetilde W_0}.
\end{equation}
Hence the equivalence of the norms. The fact that the norm $\|\cdot\|_{\widetilde W_0}$ is equivalent to the norm $\|\cdot\|_{W_n}$ on $W_0$ is obvious and yields the desired result.
\end{proof}
We will make use of these norms in the proof of Lemma~\ref{lem:linear}. Lastly, we introduce the following Lemma which is used in the proof of Lemma~\ref{lem:strong_uniqueness}.
\begin{lemma}\label{lem:reg}
Let us assume that $d\leq 3$. For all $\gamma >0$, there exists a constant $C_\gamma >0$ such that 
\begin{equation}\label{eq:ineqimport}
\forall v \in H^2(\Omega), \quad \|v\|^2_{L^\infty(\Omega)} \leq \gamma \|\Delta v \|^2_{L^2(\Omega)} + C_\gamma\left( \|v\|_{L^2(\Omega)}^2 + \|\nabla v \|_{L^2(\Omega)}^2\right).
\end{equation}
\end{lemma}
\begin{proof}[Proof of Lemma~\ref{lem:reg}]
To prove (\ref{eq:ineqimport}), we use the continuity of the embeddings $H^1(\Omega) \hookrightarrow  L^{p}(\Omega)$ for all $p\leq 6$ and $W^{1,3 + \delta}(\Omega) \hookrightarrow L^\infty(\Omega)$ for all $\delta >0$. 
Thus, for all $v\in H^2(\Omega)$, it holds that $v \in  W^{1,3 + 1/4}(\Omega)$. Throughout the proof, $C$ will denote a positive constant, which is independent of $v$, and may change along computatations.
It holds that
\begin{align*}
\|v\|_{L^{\infty}(\Omega)}^2 & \leq C \|v\|_{W^{1,3 + 1/4}(\Omega)}^2\\
&= C \left( \|v\|_{L^{3 + 1/4}(\Omega)} + \|\nabla v\|_{L^{3 + 1/4}(\Omega)}\right)^2\\
& \leq C \left( \|v\|_{H^1(\Omega)} +  \|\nabla v\|_{L^{3+1/4}(\Omega)}\right)^2.\\
\end{align*} 
To estimate the term $\|\nabla v\|_{L^{3 +1/4}(\Omega)}$, we now use the Gagliardo-Nirenberg-Sobolev inequality \cite[Thm 13.54]{Leoni2009} to obtain that
$$
\|\nabla v\|_{L^{3 +1/4}(\Omega)} \leq C \left( \|\Delta v\|_{L^2(\Omega)}^{1/2} \|v\|^{1/2}_{L^{13/3}(\Omega)}  + \|v\|_{L^{13/3}(\Omega)}\right) 
\leq C \left( \|\Delta v\|_{L^2(\Omega)}^{1/2} \|v\|^{1/2}_{H^1(\Omega)} + \|v\|_{H^1(\Omega)}\right).
$$
Thus, we obtain that for all $\gamma >0$, 
\begin{align*}\label{eq:w13estimate}
\|v\|_{L^\infty(\Omega)}^2 &\le C\left( \|v\|_{H^1(\Omega)}+ \|v\|_{H^1}^{1/2}\|\Delta v \|_{L^2(\Omega)}^{1/2}) \right)^2\\
&\le C\left( \|v\|_{H^1(\Omega)}+ \frac{C}{2\gamma}\|v\|_{H^1} + \frac{\gamma}{2C}\|\Delta v \|_{L^2(\Omega)} \right)^2. \\
&\le C_\gamma (\|\nabla v\|_{L^2(\Omega)}^2+\|v\|_{L^2(\Omega)}^2)+ \gamma\|\Delta v \|_{L^2(\Omega)}^2,\\
\end{align*}
where $C_\gamma>0$ is a positive constant which depends on $\gamma$ but is independent of $v$. Hence the desired result.
\end{proof}

\REV{\subsection{Uniqueness of strong solutions}}

To \REVTWO{improve} readability, we first show the uniqueness of strong solutions. 
\begin{lemma}[Uniqueness of strong solutions]\label{lem:strong_uniqueness}
	Let \REVV{$d\leq 3$ and let }us assume that Assumptions~\ref{ass:A1} and~\ref{ass:A2} hold. Let $u^0\in [H^1(\Omega)]^{n+1}$, with $u^0(x) \in\overline{\mathcal{P}}$ for almost all $x\in\Omega$. 
	If there exists at least one strong solution $u$ to \eqref{eq:crossnplusone2} (or equivalently to \eqref{eq:crossnplusone}) such that
	\begin{itemize}
		\item [(i)] $\displaystyle u \in [L^2((0,T), H^2(\Omega)) \cap H^1((0,T), L^2(\Omega))]^{n+1}$,
		\item[(ii)] $u(t,x)\in \overline{\cP}$ for almost all $(t,x)\in (0,T)\times \Omega$,
	\end{itemize}
	then it is unique.
\end{lemma}

\REVV{
\begin{proof}
Let $\left(u_{i,1}\right)_{0\leq i \leq n}$ and $\left(u_{i,2}\right)_{0\leq i \leq n}$ be two strong solutions to \eqref{eq:crossnplusone2} satisfying (i) and (ii) with initial datum $u^0$ 
and let us denote by $\bar u_i = u_{i,1} - u_{i,2}$. \REVTWO{The equation for the $i^{th}$ read as 
\begin{equation}\label{eq:icross}
\partial_t u_{i,k}= \sum_{j=0, j\neq i}^n K_{ij}(u_{j,k}\Delta u_{i,k}- u_{i,k} \Delta u_{j,k}), \quad k=1,2.
\end{equation}
Using the fact that $\sum_{j=0}^n u_{j,} = 1$, \REV{which implies that
	\begin{align*}
	\Delta u_{i,k} = \sum_{j=0}^n u_{j,k} \Delta u_{i,k} = \sum_{j=0}^n (u_{j,k} \Delta u_{i,k} - u_{i,k} \Delta u_{j,k}) = \sum_{j=0, j \neq i}^n (u_{j,k} \Delta u_{i,k} - u_{i,k} \Delta u_{j,k}),
	\end{align*}
}
and subtracting the equations for $k=1,2$ we obtain that for all $0\leq i \leq n$},
\begin{equation}\label{eq:suniqueness_diff}
\partial_t \bar u_i - K \Delta \bar u_i = \sum_{j=0, j\neq i}^n(K_{ij}-K)(u_{j,2}\Delta \bar u_i + \bar u_j \Delta u_{i,1} - (u_{i,2} \Delta \bar u_j + \bar u_i \Delta u_{j,1})).
\end{equation}
As a first step we multiply (\ref{eq:suniqueness_diff}) by $-\Delta \bar u_i$ and integrate over $\Omega$. This yields, almost everywhere in $(0,T)$ and for all $\eta>0$,
\begin{align*}
 &\frac{d}{dt}\int_\Omega |\nabla \bar u_i|^2 + K \int_\Omega(\Delta \bar u_i)^2   = \sum_{j=0, j\neq i}^n(K_{ij}-K)\int_\Omega (u_{j,2}(\Delta \bar u_i)^2 + \bar u_j \Delta u_{i,1}\Delta \bar u_i - u_{i,2} \Delta \bar u_j\Delta \bar u_i 
 - \bar u_i \Delta u_{j,1}\Delta \bar u_i)\\
 &\le n\kappa \int_\Omega (\Delta \bar u_i)^2 + \frac{\kappa}{2\eta}\sum_{j=0,j\neq i}^n \|\bar u_j\|_{L^\infty(\Omega)}^2 + \frac{n\kappa\eta}{2}\|\Delta u_{i,1}\|_{L^2(\Omega)}^2\|\Delta \bar u_i\|_{L^2(\Omega))}^2\\
  &+\kappa  \left(\frac{n}{2}\|\Delta \bar u_i\|_{L^2(\Omega)}^2 + \sum_{j=0,j\neq i }^n\frac{1}{2}\|\Delta \bar u_j\|_{L^2(\Omega)}^2 + \frac{n}{2\eta}\|\bar u_i\|_{L^\infty(\Omega)}^2 + \frac{\eta}{2} \|\Delta \bar u_i\|_{L^2(\Omega)}^2 \sum_{j=0,j\neq i }^n\|\Delta u_{j,1}\|_{L^2(\Omega)}^2\right).
 \end{align*}%
 Let us now choose $\eps >0$ small enough and define $\eta_\eps(t) = \frac{\eps}{1+ n\kappa \sum_{i=0}^n \|\Delta u_{i,1}(t,\cdot)\|_{L^2(\Omega)}^2}$. 
 Choosing $\eta = \eta_\eps(t)$ in the above inequality and summing over $i=0,\ldots, n$, we obtain that for almost all $t\in (0,T)$:
 \begin{equation}\label{eq:suniqueness_1}
  \sum_{i=0}^n\left(\frac{d}{dt}\int_\Omega |\nabla \bar u_i|^2\;dx + (K- 2n\kappa - \eps) \int_\Omega(\Delta \bar u_i)^2\;dx\right)  \le \frac{n\kappa}{\eta_\eps} \sum_{i=0}^n\|\bar u_i\|_{L^\infty(\Omega)}^2.
 \end{equation}%
 Let us point out that $\frac{1}{\eta_\eps} \in L^1(0,T)$ by definition. To estimate the terms on the right hand side, we make use of Lemma~\ref{lem:reg} and obtain that for all $\gamma >0$, there exists $C_\gamma >0$ such that for all $0\leq i \leq n$, 
 \begin{equation}\label{eq:auxineq}
 \|\bar u_i\|_{L^{\infty}(\Omega)}^2 \leq \gamma \|\Delta \bar u_i \|_{L^2(\Omega)}^2 + C_\gamma \left( \|\nabla \bar u_i \|_{L^2(\Omega)}^2 + \|\bar u_i\|_{L^2(\Omega)}^2\right)
 \end{equation}
 Thus, choosing $\gamma = \eps$ in (\ref{eq:auxineq}), we obtain 
 \begin{equation}\label{eq:suniqueness_2}
  \sum_{i=0}^n\left(\frac{d}{dt}\int_\Omega |\nabla \bar u_i|^2 + (K- 2n\kappa - 2\eps) \int_\Omega(\Delta \bar u_i)^2\right)  \le g_\eps \sum_{i=0}^n\left(\|\bar u_i\|_{L^2(\Omega)}^2 + \|\nabla \bar u_i\|_{L^2(\Omega)}^2 \right),
 \end{equation}%
for some function $g_\eps\in L^1(0,T)$. In order to control the $\|\bar u_i\|_{L^2(\Omega)}^2$-term on the right hand side, we consider the weak form of \eqref{eq:suniqueness_diff} and chose $\bar u_i$ as a test function. We obtain
 \begin{align*}
   & \frac{d}{dt} \int_\Omega \left|{\bar u_i}\right|^2\;dx + K \int_\Omega |\nabla \bar u_i|^2\;dx = \\
   & \sum_{j=0, j\neq i}^n(K_{ij}-K)\int_\Omega 2 \bar u_i\nabla u_{j,1} \cdot \nabla\bar  u_i + \nabla \bar u_j \cdot (\bar u_i \nabla u_{i,2} + u_{i,2} \nabla \bar u_i) \\
   & - \sum_{j=0, j\neq i}^n(K_{ij}-K)\int_\Omega \nabla \bar u_i\cdot(\bar u_i \nabla u_{j,2} + u_{j,2} \nabla \bar u_i) + \nabla u_{i,1} \cdot (\bar u_i \nabla \bar u_j + \bar u_j \nabla \bar u_i )\\
    & = \sum_{j=0, j\neq i}^n(K_{ij}-K)\int_\Omega \bar u_i\nabla u_{j,1} \cdot \nabla\bar  u_i  - u_{j,2} |\nabla \bar u_i|^2 + u_{i,2} \nabla \bar u_i \cdot \nabla \bar u_j - \bar u_j \nabla u_{i,1} \cdot \nabla \bar u_i\\
 \end{align*}
 Arguing as above and using again Lemma~\ref{lem:reg}, we eventually obtain that for all $\eps>0$, there exists a function $\widetilde{g}_\eps\in L^1(0,T)$ such that
 \begin{align}\label{eq:suniqueness_3}
  &\sum_{i=0}^n\left(\frac{d}{dt}\int_\Omega \left|\bar u_i\right|^2 + (K- 2n\kappa - \eps) \int_\Omega|\nabla  \bar u_i|^2 - \eps \int_\Omega |\Delta \bar u_i|^2\right) \le 
  \widetilde{g}_\eps \sum_{i=0}^n \left(\|\bar u_i\|_{L^2(\Omega)}^2 + \|\nabla \bar u_i\|_{L^2(\Omega)}^2\right).
 \end{align}
 Adding \eqref{eq:suniqueness_3} and \eqref{eq:suniqueness_2} and choosing $\eps$ small enough thus imply that
 \begin{align*}
  &\sum_{i=0}^n\left(\frac{d}{dt}\int_\Omega \bar u_i^2\;dx + \frac{d}{dt}\int_\Omega |\nabla\bar u_i|^2\;dx\right) \le h \sum_{i=0}^n \left(\|\bar u_i\|_{L^2(\Omega)}^2 + \|\nabla \bar u_i\|_{L^2(\Omega)}^2\right), 
 \end{align*}
 for some function $h \in L^1(0,T)$. As the initial data of $(u_{i,1})_{0\leq i \leq n}$ and $(u_{i,2})_{0\leq i \leq n}$ coincide in $H^1(\Omega)$, Gronwall's lemma yields the assertion.
 \end{proof}
 }

\subsection{Existence for a linear problem}

To prove Theorem~\ref{thm:existencestrong}, we begin by proving the existence of a strong solution to a truncated linearized approximate problem, which we present hereafter. 

\medskip

\REVV{
In this section and the following one, we fix the initial condition $u^0 = (u_0^0, \cdots, u_n^0)\in H^1(\Omega)^{n+1}$, and define for all $0\leq i \leq n$
$$
Z_i:= \left\{ u\in W, \; \nabla u \cdot \textbf{n} = 0, \; u(0) = u_i^0 \right\},
$$
and $Z:= Z_0\times \cdots \times Z_n$. 
}

\medskip

Let us assume for now that there exists a smooth solution $u:=(u_0, \ldots, u_n)$ to~\eqref{eq:cross} satisfying $\sum_{j=0}^n u_j = 1$.
Using again the fact that $\sum_{j=0}^n u_j = 1$, the equation for each component (\ref{eq:icross}) can be rewritten as
\begin{equation} \label{eq:strongicross} 
 \partial_t u_i - K \Delta u_i = \sum_{j=0, j\neq i}^n(K_{ij}-K)(u_j\Delta u_i- u_i \Delta u_j),
\end{equation}
where the \REV{the positive constant $K$ is }defined in (\ref{eq:defKi}).

\medskip


Let $\tilde u :=(\tilde u_0, \tilde u_1, \ldots, \tilde u_n)\in \REVV{Z}$. We consider the following linear, regularized problem:

\REV{\begin{equation}\label{eq:cross_truncated}
\left\{
\begin{array}{rll}
\partial_t u_i - K\Delta u_i & = \sum_{j=0, j\neq i}^n(K_{ij}-K)({\tilde u_j^\diamond}\Delta u_i- {\tilde u_i^\diamond} \Delta u_j),& \\
\nabla u_i \cdot \textbf{n} & = 0 ,& \\
\end{array} \right.
\quad i=0,\ldots, n,
\end{equation}
with $x^\diamond:= \max(0, \min(1,x))$. Note that this implies $0 \le x^\diamond \le 1$ for every $x \in \mathbb{R}$.
}
Also note that even though we are dealing with a linear problem, we employ a fixed point strategy as in \cite{Breden2017}.
\begin{lemma}[Existence of a strong solution to the linearized problem]\label{lem:linear}
\REVV{For all $\tilde u:=(\tilde u_0,\ldots, \tilde u_n) \in Z$, there exists a unique solution 
$u:=(u_0,\ldots, u_n) \in Z$ to \eqref{eq:cross_truncated}.} In addition, the three following a priori estimates hold
\begin{align}\label{eq:ineq1}
\sum_{i=0}^n\mathop{\sup}_{0\leq t \leq T} \|\nabla u_i(\cdot, t)\|_{L^2(\Omega)}^2 + \sum_{i=0}^n\int_0^T\int_\Omega (\Delta u_i)^2 & 
\leq C_0,\\ \label{eq:ineq2}
\mathop{\sup}_{0\leq s \leq T}\sum_{i=0}^n \left\| u_i(t,\cdot)\right\|^2_{L^2(\Omega)} & \leq C_1,\\ \label{eq:ineq3}
\sum_{i=0}^n \left\| \partial_t u_i\right\|^2_{L^2(0,T; L^2(\Omega))} & \leq C_2,\\ \nonumber
\end{align}
with 
\begin{align} \label{eq:C0}
 C_0 & := 2 \max\left(1, \frac{1}{K - 2n\kappa }\right) \sum_{i=0}^n\|\nabla u_i^0\|_{L^2(\Omega)}^2,\\ \label{eq:C1}
 C_1 & := e^{2n \kappa  T}\left( \sum_{i=0}^n \left\| u_i^0\right\|^2_{L^2(\Omega)} + 2n \kappa  C_0 \right),\\ \label{eq:C2}
 C_2& := \left( K + 2n\kappa \right) C_0.
\end{align}
\end{lemma}

\begin{proof}
\REV{
\itshape Step 1 (Existence): \normalfont 
Let $\tilde u:=(\tilde u_0,\ldots, \tilde u_n) \in Z$. For all $(\overline u_0, \ldots, \overline u_n) \in \REVV{Z}$, consider the problem 
\begin{equation}\label{eq:cross_truncated2}
\left\{
\begin{array}{rll}
\partial_t u_i - K\Delta u_i & = \sum_{j=0, j\neq i}^n(K_{ij}-K)({\tilde u_j^\diamond}\Delta \overline u_i- {\tilde u_i^\diamond} \Delta \overline u_j),& \\
\nabla u_i \cdot \textbf{n} & = 0 ,& \\
\end{array} \right.
\quad i=0,\ldots, n.
\end{equation}
Now as $\Delta \overline u_i \in L^2((0,T);L^2(\Omega))$ and $\tilde u_i^\diamond \in L^\infty((0,T);L^\infty(\Omega))$ for $i=0,\ldots, n$, standard theory for linear parabolic equations yields 
the existence of a \REVV{unique} solution $(u_i)_{0\leq i \leq n} \in \REVV{Z}$. Let us denote by
$$
F: \left\{
\begin{array}{lll}
 Z & \to & Z\\
 (\bar u_{i})_{0\leq i \leq n} & \mapsto & (u_i)_{0\leq i \leq n}\\
\end{array}
\right.
$$
the application such that $(u_i)_{0\leq i \leq n} \in Z$ is the unique solution to (\ref{eq:cross_truncated2}). 
}

\REV{
Let now $(\overline u_i^1)_{0\leq i \leq n},(\overline u_i^2)_{0\leq i \leq n}\in Z$ and let $(u_i^1)_{0\leq i \leq n}:= F\left( (\bar u_i^1)_{0\leq i \leq n}\right)$ and $(u_i^2)_{0\leq i \leq n}:= F\left( (\bar u_i^2)_{0\leq i \leq n}\right)$.
Then, we have
\begin{equation*}
\left\{
\begin{array}{rll}
\partial_t (u_i^1-u_i^2) - K\Delta (u_i^1-u_i^2) & = \sum_{j=0, j\neq i}^n(K_{ij}-K)({\tilde u_j^\diamond}\Delta  (\overline u_i^1-\overline u_i^2)- {\tilde u_i^\diamond} \Delta (\overline  u_j^1-\overline u_j^2)),& \\
\nabla (u_i^1-u_i^2) \cdot \textbf{n} & = 0 ,& \\
(u_i^1-u_i^2)(0,\cdot) & = 0
\end{array} \right.
\quad i=0,\ldots, n.
\end{equation*}
Taking the $L^2((0,T);L^2(\Omega))$-norm on both sides, noting that both $\bar u_i^1 - \bar u_i^2 \in W_0$ and $u_i^1 - u_i^2 \in W_0$, and (\ref{eq:normestimate}), yields
\begin{align*}
\|u_i^1 - u_i^2 \|_{\widetilde W_0} \le \|u_i^1 - u_i^2 \|_{W_0} &\le \kappa n \|\Delta(\overline u_i^1 - \overline u_i^2)\|_{L^2(\Omega \times (0,T))} + \kappa \sum_{j=0, j\neq i}^n \|\Delta (\overline  u_j^1-\overline u_j^2)\|_{L^2(\Omega \times (0,T))}\\
&= \frac{\kappa n}{K} K\|\Delta(\overline u_i^1 - \overline u_i^2)\|_{L^2(\Omega \times (0,T))} + \sum_{j=0, j\neq i}^n \frac{\kappa }{K}K\|\Delta (\overline  u_j^1-\overline u_j^2)\|_{L^2(\Omega \times (0,T))},
\end{align*}
\REVTWO{where we used that by Assumption \ref{ass:A2} we have $|K_{ij} - K| \le \kappa$ and by definition $\tilde u_i^\diamond \le 1$.} Summing over $i=0,\ldots,n$, we obtain
\begin{align}
\begin{split}
\sum_{i=0}^n \|u_i^1 - u_i^2 \|_{\widetilde W_0} &\le  \sum_{i=0}^n\left(\frac{\kappa n}{K} K\|\Delta(\overline u_i^1 - \overline u_i^2)\|_{L^2(\Omega \times (0,T))} + \sum_{j=0, j\neq i}^n \frac{\kappa }{K}K\|\Delta (\overline  u_j^1-\overline u_j^2)\|_{L^2(\Omega \times (0,T))}\right)\\ 
&\le \underbrace{\frac{2\kappa n}{K}}_{< 1} \sum_{i=0}^n \|\bar u_i^1 - \bar u_i^2\|_{\widetilde W_0}.
\end{split}\label{eq:contraction}
\end{align}

\REVV{
Let us now introduce the distance $d:Z\times Z \to \R_+$ defined by
$$
\forall u:=(u_i)_{0\leq i \leq n},\, v:=(v_i)_{0\leq i \leq n} \in Z, \quad d(u,v):= \sum_{i=0}^n \|u_i - v_i\|_{\widetilde W_0}.
$$
Then, $(Z,d)$ is a complete metric space, and (\ref{eq:contraction}) implies that the map $F$ is a contraction with respect to $d$. 
Banach's fixed point theorem then ensures the existence and uniqueness of a strong solution $ u = (u_0,\ldots, u_n) \in Z$ to the equation \eqref{eq:cross_truncated}.}
}

\medskip

\itshape Step 2 (A priori estimates): \normalfont 
Denoting again by $u:=(u_0, \ldots, u_n)$, we can rewrite the system (\ref{eq:cross_truncated}) as follows:
\begin{align}\label{eq:matrix}
\partial_t u = (P-B(\tilde u) )\Delta u
\end{align}
where $P= KI$ with $I$ the identity matrix in $\R^{(n+1)\times (n+1)}$ and
\begin{align*}
B(\tilde u):=\begin{pmatrix}
\displaystyle \sum_{j=0, j\neq 0}^n(K_{0j}-K)\tilde{u}_j^\diamond & \dots & -(K_{0n}-K)\tilde{u}_0^\diamond\\
\vdots & \ddots & \vdots \\
-(K_{n0}-K)\tilde{u}_n^\diamond & \dots &   \sum_{j=0, j\neq n}^n(K_{nj}-K)\tilde{u}_j^\diamond\\
\end{pmatrix}.
\end{align*}
For any $\xi \in \RR^{n+1}$, we have, using Assumption \ref{ass:A2},
\begin{align*}
 \xi^T (P-B(\tilde u)) \xi &= \sum_{i=0}^n K \xi_i^2 + \sum_{i=0}^n\sum_{j=0,i\neq j}^n (K_{ij}-K)\left( \tilde u_j^\diamond \xi_i^2 - \tilde u_i^\diamond \xi_i\xi_j\right)\\
 &\ge \left(\displaystyle K - n\kappa \right) \|\xi\|_2^2 - \sum_{i=0}^n\sum_{j=0,i\neq j}^n (K_{ij}-K) \tilde u_i^\diamond \xi_i\xi_j.\\
\end{align*}
Besides, using again Assumption \ref{ass:A2}, it holds that
\begin{align*}
 &\left| \sum_{i=0}^n \sum_{j=0,i\neq j}^n (K_{ij} - K)\tilde u_i^\diamond \xi_i\xi_j\right| \le \sum_{i=0}^n \sum_{j=0,i\neq j}^n\kappa |\tilde u_i^\diamond| |\xi_i\xi_j|\\
 &\le\frac{1}{2}\sum_{i=0}^n\sum_{j=0,i\neq j}^n\kappa \left(\xi_i^2+\xi_j^2\right)
 \le \frac{n}{2}\kappa  \|\xi\|_2^2 + \frac{n}{2}\kappa  \|\xi\|_2^2.
\end{align*}
We thus obtain
\begin{align}\label{eq:coercivity}
 \xi^T (P-B(\tilde u)) \xi \ge \left( K \; - 2n\kappa  \right) \|\xi\|_2^2.
\end{align}
Now multiplying \eqref{eq:matrix} by the vector $(-\Delta u_0, \ldots, -\Delta u_n)^T$ and integrating over $\Omega$, we obtain that for almost all $t\in (0,T)$,
\begin{align*}
-\sum_{i=0}^n(\partial_t u_i, \Delta u_i)_{L^2(\Omega)} +  \left( K - 2n\kappa  \right) \sum_{i=0}^n\int_\Omega (\Delta u_i)^2\;dx \le 0,
\end{align*}
which implies, using Lemma \ref{lem:timederiv} and integrating in time between $0$ and $t$, 
$$
\sum_{i=0}^n\|\nabla u_i(\cdot, t)\|_{L^2(\Omega)}^2 + \left( K - 2n\kappa  \right)\sum_{i=0}^n\int_0^t\int_\Omega (\Delta u_i)^2\;dxdt 
\le \sum_{i=0}^n\|\nabla u_i^0\|_{L^2(\Omega)}^2.
$$
We thus get
\begin{align}\label{eq:h2linear}
\sum_{i=0}^n\sup_{0\leq t \leq T} \|\nabla u_i(\cdot, t)\|_{L^2(\Omega)}^2 + &\left( K - 2n\kappa  \right)\sum_{i=0}^n\int_0^T\int_\Omega (\Delta u_i)^2\;dxdt
\le 2 \sum_{i=0}^n\|\nabla u_i^0\|_{L^2(\Omega)}^2, 
\end{align}
which immediately yields (\ref{eq:ineq1}).
On the other hand, multiplying \eqref{eq:matrix} by the vector $(u_0, \ldots, u_n)^T$ and integrating over $\Omega$, we obtain that for almost all $t\in (0,T)$,
\begin{align*}
\sum_{i=0}^n\frac{d}{dt}\left( \left\| u_i(t,\cdot)\right\|^2_{L^2(\Omega)}\right) +\sum_{i=0}^n\int_\Omega K |\nabla u_i|^2\;dx & = \sum_{i=0}^n \sum_{j=0, j\neq i}^n (K_{ij}-K)\int_\Omega ({\tilde u_j^\diamond}u_i \Delta u_i- {\tilde u_i^\diamond}u_i \Delta u_j)\;dx \\
& \leq  2n \kappa  \sum_{i=0}^n \left( \|\Delta u_i(\cdot, t)\|_{L^2(\Omega)}^2 +  \|u_i(t,\cdot)\|_{L^2(\Omega)}^2 \right).\\
\end{align*}
Applying Gronwall's lemma, we thus obtain that 
\begin{equation}
\mathop{\sup}_{0\leq s \leq T}\sum_{i=0}^n \left\| u_i(t,\cdot)\right\|^2_{L^2(\Omega)} \leq e^{2n \kappa  T}\left( \sum_{i=0}^n \left\| u_i^0\right\|^2_{L^2(\Omega)} + 2n \kappa  \|\Delta u_i(\cdot, t)\|_{L^2(0,T; L^2(\Omega))}^2 \right),
\end{equation}
which yields (\ref{eq:ineq2}).
Lastly, using (\ref{eq:cross_truncated}), we obtain that
$$
\sum_{i=0}^n \left\| \partial_t u_i \right\|_{L^2(0,T, L^2(\Omega))}^2 \leq \left(  K + 2n \kappa  \right) \sum_{i=0}^n \left\| \Delta u_i \right\|^2_{L^2(0,T; L^2(\Omega))}, 
$$
which immediately yields estimate (\ref{eq:ineq3}).
\end{proof}


\subsection{Proof of Theorem \ref{thm:existencestrong}}

\begin{proof}[Proof of Theorem \ref{thm:existencestrong}]

Let us now assume that $u^0:=(u_0^0, u_1^0, \cdots, u_n^0) \in H^1(\Omega)^{n+1}$ satsifies $u^0(x)\in \overline{P}$ for almost all $x\in \Omega$ and use the same notation as in the preceding section. 

\medskip

\REV{Let us denote by $\mathcal{M}$ the set of functions $(u_0, \ldots, u_n)\in \REVV{Z}$ satisfying (\ref{eq:ineq1}), (\ref{eq:ineq2}) and (\ref{eq:ineq3})
with constants $C_0$, $C_1$ and $C_2$ defined by (\ref{eq:C0}), (\ref{eq:C1}) and (\ref{eq:C2}) respectively. For all $(\tilde u_0, \ldots, \tilde u_n)\in \mathcal{M}$, let us denote by 
$S\left( (\tilde u_0, \ldots, \tilde u_n) \right):= (u_0, \ldots, u_n)\in \REVV{Z}$, where $(u_0, \ldots, u_n)$ is the unique strong solution of (\ref{eq:cross_truncated}). 
}

\medskip

In view of Lemma \ref{thm:existencestrong}, the operator $S: \mathcal{M} \to \REVV{Z}$ is well-defined and self-mapping, i.e. $S(\mathcal{M}) \subset \mathcal{M}$. 
Moreover, due to the Aubin-Lions lemma \cite[Theorem~5.1, p. 58]{lions1969quelques}, the set $\mathcal{M}$ is a convex compact subset of $L^2((0,T);L^2(\Omega))$. 
\REV{In order to apply Brouwer's fixed point theorem, it remains to show that $S$ is continuous. We consider a sequence $(\tilde u^\delta_0, \ldots, \tilde u_n^\delta)_{\delta >0} \subset \mathcal{M}$ which strongly 
converges in $L^2((0,T);L^2(\Omega))$ to some $(\tilde u_0, \ldots, \tilde u_n)\in \mathcal{M}$. Thus if for all $\delta >0$, we denote by $(u_0^\delta, \ldots, u_n^\delta)\in \REVV{Z}$ the unique solution to
\begin{align}\label{eq:cross_truncated_regularized}
\begin{split}
\partial_t u^\delta_i(x,t) - K\Delta u^\delta_i&= \sum_{j=0, j\neq i}^n(K_{ij}-K)(({\tilde u^\delta_j)^\diamond}\Delta u^\delta_i- ({\tilde u^\delta_i})^\diamond \Delta u^\delta_j),\\
\nabla u^\delta_i \cdot \textbf{n} &= 0,\\ 
\end{split}
\quad 
i=0,\ldots,n.
\end{align}
with initial condition $(u_0^0, \ldots, u_n^0)$. Using the a priori estimates \eqref{eq:ineq1}, \eqref{eq:ineq2} and \eqref{eq:ineq2} we  obtain that the sequence
$(u_0^\delta, \ldots, u_n^\delta)_{\delta >0}$ is thus bounded in $W^{n+1}$. Up to the extraction of a subsequence, there exists $(u_0, \ldots,u_n)\in  W^{n+1}$ such that $(u_i^\delta)_{\delta>0}$ weakly 
converges in $W$ to $u_i$ for all $0\leq i \leq n$. In addition we have that $(\tilde u^\delta_i)^\diamond \to \tilde u_i^\diamond$ as the mapping $u \mapsto u^\diamond$ is Lipschitz continuous with Lipschitz constant $1$. 
This allows us to pass to the limit $\delta \to 0$ in (\ref{eq:cross_truncated_regularized}) in the distributional sense, and yields the continuity of $S$.
}
\medskip

Thus, we can apply Brouwer's fixed point theorem and conclude to the existence of a strong solution $(u_0,\ldots, u_n) \in W^{n+1}$ to the regularized system
\REV{
\begin{equation}\label{eq:cross_truncated3}
\left\{
\begin{array}{rll}
\partial_t u_i - K\Delta u_i & = \sum_{j=0, j\neq i}^n(K_{ij}-K)(u_j^\diamond\Delta u_i- u_i^\diamond \Delta u_j),& \\
\nabla u_i \cdot \textbf{n} & = 0 ,& \\
u_i(0,\cdot) & = u_i^0,  \\
\end{array} \right.
\quad i=0,\ldots, n,
\end{equation}
which satisfies the a priori estimates
\begin{align*}
\sum_{i=0}^n\mathop{\sup}_{0\leq t \leq T} \|\nabla u_i(\cdot, t)\|_{L^2(\Omega)}^2 + \sum_{i=0}^n\int_0^T\int_\Omega (\Delta u_i)^2 & 
\leq C_0,\\ 
\mathop{\sup}_{0\leq s \leq T}\sum_{i=0}^n \left\| u_i(t,\cdot)\right\|^2_{L^2(\Omega)} & \leq C_1,\\ 
\sum_{i=0}^n \left\| \partial_t u_i\right\|^2_{L^2(0,T; L^2(\Omega)}) & \leq C_2,\\ 
\end{align*}
where $C_0$, $C_1$ and $C_2$ are defined respectively in (\ref{eq:C0}), (\ref{eq:C1}) and (\ref{eq:C2}). 
}

\medskip

To end the proof, it remains to show that $u_i\geq 0$ almost everywhere in $(0,T)\times \Omega$ for all $0\leq i \leq n$. 
\REV{First note that $(u_0,\cdots,u_n)$ satisfies the system of equations
\begin{align}\label{eq:strong_linear}
 \left\{
 \begin{array}{l}
  \partial_t u_i - A_i(x,t) \Delta u_i = B_i(x,t) u_i^\diamond,\\
  \nabla u_i \cdot \textbf{n} = 0 \; \mbox{ on }(0,T)\times \partial \Omega,\\
  u_i(t=0,\cdot) = u_i^0,\\ 
 \end{array}
 \right.
\; i =0,\ldots, n,
\end{align}
with 
\begin{align}
A_i(x,t) := \left(K- \sum_{j=0, j\neq i}^n(K_{ij}-K)u^\diamond_j\right)\text{ and } B_i(x,t) :=\left(-\sum_{j=0, j\neq i}^n(K_{ij}-K) \Delta u_j\right).
\end{align}
In particular, if we consider $(A_i)_{0\leq i \leq n}$ and $(B_i)_{0\leq i \leq n}$ as given coefficients, \REVV{it holds that there exists a unique solution $(u_0, \cdots, u_n)\in Z$ to \eqref{eq:strong_linear}. 
Indeed, this can be shown using the same arguments as in the proof} of Lemma~\ref{lem:strong_uniqueness}, i.e. testing both with the difference of two solutions 
as well as the Laplace of that difference and using again Lemma~\ref{lem:reg}. 
}

\REV{
In order to show the desired non-negativity, we regularise the coefficients $A_i(x,t)$ (with respect to to the $x$ variable) by convolving it with a smooth kernel. \REVV{More precisely, 
let $\eta \in \cC^\infty(\R^d)$ be a standard non-negative mollifier so that $\int_{\R^d}\eta  = 1$ and for all $\eps>0$, let us denote by $\eta_\eps(x):=\frac{1}{\eps^d}\eta(x/\eps)$. 
Note that $K + n \kappa \geq A_i(x,t) \geq K - n\kappa$ a.e. in $(0,T)\times \Omega$. We extend $A_i(t,x)$ to a function defined over $(0,T)\times \R^d$ by defining
$$
\overline{A}_i(x,t):=\left\{
\begin{array}{ll}
 A_i(t,x) & \mbox{ if } x\in \Omega,\\
 K & \mbox{ otherwise}.\\
\end{array}
\right.
$$
We then define for all $\eps>0$,
$$
A_i^\eps(t,x):= \int_{\Omega} \overline{A}_i(t,y) \eta_\eps(x-y)\,dy, \quad \forall (t,x)\in (0,T)\times \Omega.
$$
Then, it holds that for all $\eps >0$, $K + n \kappa \geq A^\eps_i(x,t) \geq K - n\kappa$ a.e. in $(0,T)\times \Omega$ and for almost all $t\in (0,T)$, 
$\displaystyle A_i^\eps(t,\cdot) \mathop{\longrightarrow}_{\eps \to 0} A_i(t,\cdot)$ strongly in $L^p(\Omega)$ for all $0<p<\infty$. }
}

We also denote by $(u_i^\eps)_{0\leq i \leq n}$ 
the (unique) solution in $Z$ to
\begin{align}\label{eq:linear_regularised}
\partial_t u_i^\eps - A_i^\eps \Delta u_i^\eps = B_i(x,t) u_i^{\eps, \diamond},\; i =0,\ldots, n.
\end{align}
The existence and uniqueness of strong solutions in $Z$ to (\ref{eq:linear_regularised}) can be obtained again using similar arguments as above.
We claim that for all $0\leq i \leq n$, $u_i^\eps \ge 0$. Indeed, multiplying (\ref{eq:linear_regularised}) with the negative part $(u_i^\eps)_-$ and integrating over the spatial domain $\Omega$ gives
\begin{align*}
 \frac{d}{dt} \int_\Omega \left|(u_i^\eps)_-\right|^2  + C\int_\Omega |\nabla (u_i^\eps)_-|^2 \le  \underbrace{\int_\Omega B_i u_i^{\diamond,\eps} (u_i^\eps)_- }_{=0} + \int_\Omega |\nabla A_i^\eps| |\nabla (u_i^\eps)_-| (u_i^\eps)_-,
\end{align*}
i.e.
\begin{align*}
 \frac{d}{dt} \int_\Omega |(u_i^\eps)_-|^2  + (C-\gamma)\int_\Omega |\nabla (u_i^\eps)_-|^2 \le  \frac{1}{4\gamma} \|\nabla A_i^\eps \|_{L^\infty(\Omega)}\int_\Omega |(u_i^\eps)_-|^2,
\end{align*}
for any $\gamma >0$, so that Gronwall's lemma implies $u_i^\eps \ge 0$ a.e. as $(u_i^0)_- = 0$ a.e. in $(0,T)\times \Omega$. Besides, there exists  $(w_i)_{0\leq i \leq n}\in W_n^{n+1}$ such that, up to the extraction of a subsequence,
\begin{align*}
 \partial_t u_i^\eps \rightharpoonup \partial_t w_i \text{ in } L^2(0,T;L^2(\Omega))\;\text{ and }\; \Delta u_i^\eps  \rightharpoonup \Delta w_i \text{ in } L^2(0,T;L^2(\Omega)),
 \end{align*}
 as well as 
 \begin{align*}
 A_i^\eps &\to A_i \text{ in } L^p((0,T)\times \Omega),\quad \text{ for every } p < \infty, \text{ since } A_i\in L^\infty((0,T)\times \Omega),\\
 u_i^\eps &\to w_i \text{ in } L^2(0,T; L^2(\Omega)), \quad \text{by compactness of the  embedding}.
\end{align*}
The last convergence implies $u_i^{\eps,\diamond} \to w_i^\diamond$ in $L^2(0,T; L^2(\Omega))$. Thus testing \eqref{eq:linear_regularised} with $C^\infty_0$-functions 
(which are dense in $L^2(\Omega)$) we can pass to the limit $\eps \to 0$ in (\ref{eq:linear_regularised}) and conclude that 
\REVV{
\begin{itemize}
 \item $(w_i)_- = 0$ a.e. in $\Omega$ for a.e. $t \in (0,T)$ for all $0\leq i \leq n$;
 \item $w_i(t=0,\cdot) = u_i^0$;
 \item $(w_0,\cdots,w_n)$ is a solution in $Z$ to (\ref{eq:strong_linear}).
\end{itemize}
}
Using the uniqueness of strong solutions in $Z$ to (\ref{eq:strong_linear}), we thus obtain that $w_i = u_i$ for all $0\leq i \leq n$, which implies that $u_i\geq 0$. \REVTWO{Finally, to show the upper bound on the $u_i$ we note that the sum $\bar u = \sum_{i=0}^n u_i$ satisfies the heat equation
	\begin{align}\label{eq:heat}
	\left\{\begin{array}{l}
	\partial_t \bar u - K\Delta \bar u = 0,\\
	\nabla \bar u \cdot \mathbf{n} = 0,\\
	\bar u(x,0) = \sum_{i=0}^n u_i^0.
	\end{array}\right.
	\end{align}
As the initial data was such that $\sum_{i=0}^n u_i^0 = 1$, the unique solution to \eqref{eq:heat} is $\bar u(x,t)=1$ for a.e. $x \in \Omega$, $t \in (0,T)$ and thus $u(t,x)\in \overline{\mathcal{P}}$ almost everywhere.}
\end{proof}

\section{Weak strong stability}\label{sec:stability}

This section is devoted to the proof of Theorem~\ref{thm:stability}, which provides a weak-strong stability result provided that there exists a strong solution $u$ to the system of interest which satisfies the additional regularity property $\nabla u \in L^2(0,T; L^\infty(\Omega))$. 
%
\begin{proof}[Proof of Theorem~\ref{thm:stability}]
We start by rewriting the $i^{th}$ component of \eqref{eq:crossnplusone2} as 
\begin{align*}
 \int_\Omega \partial_t u_i\varphi \;dx + K\int_\Omega \nabla u_i \cdot \nabla \varphi\;dx = \int_\Omega [  \sum_{j=1,i\neq j}^n (K_{ij}-K)(u_j\nabla u_i - u_i\nabla u_i)] \cdot \nabla \varphi\;dx,
\end{align*}
for all $\varphi \in H^1(\Omega)$. Denoting by 
\begin{align*}
D(v):=\begin{pmatrix}
 \sum_{j=1}^n(K_{0j}-K)v_j& \dots & -(K_{0n}-K)v_0\\
\vdots & \ddots & \vdots \\
-(K_{n0}-K) {v_n} & \dots &   \sum_{j=0 }^{n-1}(K_{nj}-K)v_j\\
\end{pmatrix}.
\end{align*}
for all $v:=(v_i)_{0\leq i \leq n}\in \overline{\cP}$, we obtain that
\begin{equation}\label{eq:weakC}
 \int \partial_tu \Phi\;dx + \int_\Omega K\nabla u\cdot \nabla \Phi \;dx = \int_{\Omega} D(u)\nabla u \cdot \nabla \Phi\;dx,\text{ for all } \Phi \in [H^1(\Omega)]^{n+1},
\end{equation}
Since we know that $\sum_{i=0}^n u_i = 1$ and that $u_i \ge 0$ for $i=0,\ldots, n$, we immediately obtain that 
\begin{align*}
 \|D(u)\|_{L^\infty(\Omega)} \le 2 n\kappa ,
\end{align*}
in the sense of the spectral matrix norm. In addition, $D: \overline{\cP} \to \R^{(n+1)\times (n+1)}$ is Lipschitz continuous, with Lipschitz constant $2n \kappa$. 
Now we consider the difference of the respective weak formulations \eqref{eq:weakC} for $u$ and $\tilde u$ and obtain
\begin{align*}
\int_\Omega \partial_t (u-\tilde u)\Phi \;dx - K \int_\Omega  \left(\nabla u - \nabla \tilde u\right)\cdot \nabla \Phi \;dx = \int_\Omega \left[D(u)\nabla u - D(\tilde u)\nabla\tilde u\right]\cdot \nabla \Phi\;dx.
\end{align*}
Taking  $\Phi = (u-\tilde{u})(t, \cdot)$ (which belongs to $H^1(\Omega)$ for almost all $t\in (0,T)$) yields
\begin{align*}
&\frac{d}{dt}\frac{1}{2}\|u-\tilde{u}\|_{L^2(\Omega)}^2+K \|\nabla(u-\tilde{u})\|_{L^2(\Omega)}^2\\
&= - \int_\Omega (D(u)-D(\tilde{u}))\nabla u\cdot \nabla(u-\tilde{u}) \;dx-\int_\Omega D(\tilde{u})\nabla(u-\tilde{u})\cdot \nabla(u-\tilde{u})\;dx
\end{align*}
Using the fact that $\|D(\tilde{u})\|_{L^{\infty}(\Omega)} \leq 2n \kappa $ on the second term of the right hand side, we obtain
\begin{align*}
&\frac{d}{dt}\frac{1}{2}\|u-\tilde{u}\|_{L^2(\Omega)}^2+(K-2n\kappa) \|\nabla(u-\tilde{u})\|_{L^2(\Omega)}^2 \leq - \int_\Omega (D(u)-D(\tilde{u}))\nabla u \cdot \nabla(u-\tilde{u})\;dx.
\end{align*}
Since $d=1$, it holds that $L^\infty(\Omega) \subset H^1(\Omega)$ with continuous injection, which implies that $\|\nabla u \|_{L^\infty(\Omega)} \in L^2(0,T)$ (since $u\in L^2((0,T), H^2(\Omega))$. 
Thus, applying the weighted Young's inequality with $0 < \epsilon <(K-2n\kappa)$ and using the Lipschitz continuity of $D$ yield 
\begin{align*}
& \frac{d}{dt}\frac{1}{2}\|u-\tilde{u}\|_{L^2(\Omega)}^2+(K-2n\kappa-\epsilon) \|\nabla(u-\tilde{u})\|_{L^2(\Omega)}^2\leq \frac{1}{4\epsilon} \|(D(u)-D(\tilde{u}))\nabla u\|_{L^2(\Omega)}^2 \\
& \leq \frac{1}{4\epsilon}\|\nabla u\|_{L^\infty(\Omega)}^2 \|D(u)-D(\tilde{u})\|_{L^2(\Omega)}^2 \\
& \leq \frac{2n\kappa}{4\epsilon}\|\nabla u\|_{L^\infty(\Omega)}^2 \|u-\tilde{u}\|_{L^2(\Omega)}^2.\\ 
\end{align*}

Applying the differential form of the Gronwall lemma then implies that there exists $C'>0$ such that for all $t\in (0,T)$, 
\begin{align*}
\|u(t,\cdot)-\tilde{u}(t,\cdot)\|_{L^2(\Omega)}^2 \leq e^{C'\|\nabla u\|^2_{L^2((0,t), L^\infty(\Omega))}}\|u(0,\cdot)-\tilde{u}(0,\cdot)\|_{L^2(\Omega)}^2,
\end{align*}
with $C' = \frac{2n\kappa}{4\epsilon}$. Hence the result.
\end{proof}
\begin{remark} Let us remark that in dimension one, a strong solution $u$ in the sense of Theorem~\ref{thm:existencestrong} necessarily satisfies $\nabla u \in L^2(0,T; L^\infty(\Omega))$ since the injection $H^2(\Omega) \hookrightarrow W^{1,\infty}(\Omega)$ is continuous. To extend this results in higher dimension, one would need to prove the existence of solutions with this additional regularity property, for instance with more regular initial data.
\end{remark}

%
\section*{Appendix: Microscopic interpretation}
Following \cite{Burger2010}, we briefly describe a lattice based modelling approach and a formal way to obtain a \eqref{eq:cross} in the limit. We start with a one-dimensional lattice on which particles of $i=1,\ldots n$ 
species can jump to neighbouring sites. Let $\mathcal{T}_h$ denote an equidistant
grid of mesh size $h$, where a cell is either empty or can be occupied by at most one particle. We denote the probability to find a particle of species $i$ at location $x$ and time $t$ by
\begin{align*}
c_i(x,t)&=P(\text{particle of species $i$ at position $x$ at time $t$}),
\end{align*}
and assume that the motion of these particles is due to two different effects: Diffusion and exchange (switching) of particles of different species. To this end, we introduce the rates
\begin{align}\label{eq:rate1}
\Pi^{+}_{c_i}&=P(\text{jump of $c_i$ from position $x$ to $x+h$ in ($t,t+\Delta t)$})\\
&=K_{i0}(1-\rho) + \sum_{j=1,\,i\neq j}^n K_{ij}c_j,\\\nonumber
\Pi^{-}_{c_i}&=P(\text{jump of $c_i$ from position $x$ to $x-h$ in ($t,t+\Delta t)$})\\
&=K_{i0}(1-\rho) + \sum_{j=1,\,i\neq j}^n K_{ij}c_j.
\end{align}
Here $K_{i0}$ is a diffusion coefficient which controls the tendency of a particle to jump to a neighboring site. Since we restrict to at most one particle per site, 
this has to be modified by a factor of $(1-\rho)$, i.e. the particle can only jump if the target site is empty. On the other hand, in order to exchange places with a particle from a different species, 
the target site has to be occupied and thus, for the second term we have to multiply the rate $K_{ij}$ with $c_j$. 

Now we consider the following cases: If $K_{i0} \gg K_{ij}$, then the probability of switching is small compared to that of diffusion and the effect of size exclusion will be essential. 
If, on the other hand $K_{i0} \ll K_{ij}$, switching will dominate and size exclusion will not play a role anymore. Note that in this case, $\rho$, which is the sum of all densities, remains constant.

Our subsequent analysis deals with the case when $K_{i0} \approx K_{ij}$, which is the most interesting. In fact, let us rewrite \eqref{eq:rate1} as follows:
\begin{align*}
\Pi^{+}_{c_i}&=K_{i0}(1-\sum_{j=1,i\neq j}^n c_j - c_i) + \sum_{j=1,\,i \neq j}^n K_{ij}c_j,\\\nonumber
&= K_{i0}(1- c_i) + \sum_{j=1,\,i\neq j}^n (K_{ij}-K_{i0})c_j.
\end{align*}
Now if $K_{i0} \approx K_{ij}$, the switching will effectively aneal the size exclusion effect. In other words, it does not make a difference whether a target site is occupied by a particle of species $j$ or if it is empty 
since in both cases, the particle at the source site can reach this target: Either by jumping to the empty cell or by switching positions. The resulting PDE can be written as 
\begin{align*}
\partial_t c_i &= \nabla \cdot (K_{i0}((1-c_i)\nabla c_i + c_i\nabla c_i + \sum_{j=1,i\neq j}^n (K_{i0}-K_{ij})(c_j \nabla c_i - c_i \nabla _j))\\
&=\nabla \cdot (K_{i0}c_i + \sum_{j=1,i\neq j}^n (K_{i0}-K_{ij})(c_j \nabla c_i - c_i \nabla c_j)),\quad i=1,\ldots, n.
\end{align*}
which reveals that we are dealing with a perturbation of the heat equations, as already entailed in \eqref{eq:heatperturbation} in the introduction.

\section*{Acknowledgements}
The work of MB has been supported by ERC via Grant EU FP 7 - ERC Consolidator Grant 615216 LifeInverse. MB and JFP acknowledge support by the German Science Foundation DFG via EXC 1003 Cells in Motion Cluster of Excellence, M\"unster. VE acknowledges support by the ANR via the ANR JCJC COMODO project. VE and JFP are grateful to the DAAD/ANR for their support via the project 57447206. Furthermore, the authors would like to thank Robert Haller-Dintelmann (TU Darmstadt) for useful discussions. We would also like to thank the anonymous referee for his very useful comments and suggestions.

\end{document}